\documentclass[12pt]{amsart}
\usepackage{amssymb,amsfonts,amsmath,amsthm,cite,verbatim}
\usepackage{xcolor}
\usepackage{graphicx}
\usepackage{float}
\usepackage{subfigure}
\usepackage{array}
\usepackage{booktabs}
\usepackage[linesnumbered,ruled,vlined]{algorithm2e}
\usepackage[left=2.5cm, right=2.5cm, top=3.5cm, bottom=3.5cm]{geometry}
\usepackage{makecell} %单元格换行
\usepackage{hyperref}

\usepackage[normalem]{ulem}
\newtheorem{theorem}{Theorem}[section]
\newtheorem{theo}[theorem]{Theorem}

\newtheorem{lem}[theorem]{Lemma}
\newtheorem{rem}[theorem]{Remark}

\newtheorem{exam}[theorem]{Example}

\makeatletter \@addtoreset{equation}{section}

\SetCommentSty{mycommfont}

\SetKwInput{KwInput}{Input}                % Set the Input
\SetKwInput{KwOutput}{Output}              % set the Output

\newcommand\mmm[1]{operations in $#1$}
\newcommand\xmod[1]{\pmod{\langle #1 \rangle}}
\newcommand\Rsta[1]{\langle #1 \rangle_R}
\newcommand\BRsta[1]{\Big\langle #1 \Big\rangle_R}

\DeclareMathOperator*{\res}{res}

\newcommand{\QQ}{\mathbb{Q}}

\newcommand{\CC}{\mathbb{C}}

\def\a{\boldsymbol{a}}
\def\e{\mathrm{e}}

\def\P{\mathcal{P}}
\def\B{\mathcal{B}}
\def\W{\mathcal{W}}
\def\V{\mathcal{V}}

\def\hM{\widehat{M}}

\def\hPhi{\widehat{\Phi}}

\title{A polynomial time algorithm for Sylvester waves when entries are bounded}

\author{Guoce Xin$^{1}$ and Chen Zhang$^{2,*}$}

\address{ $^{1,2}$School of Mathematical Sciences, Capital Normal University, Beijing 100048, PR China}

\email{$^1$\texttt{guoce\_xin@163.com}\ \& $^2$\texttt{ch\_enz@163.com}}

\date{\today}

\thanks{$*$ Corresponding author.}

\begin{document}

\begin{abstract}
The Sylvester's denumerant \( d(t; \boldsymbol{a}) \) is a quantity that counts the number of nonnegative integer solutions to the equation \( \sum_{i=1}^{N} a_i x_i = t \), where \( \boldsymbol{a} = (a_1, \dots, a_N) \) is a sequence of distinct positive integers with \( \gcd(\boldsymbol{a}) = 1 \). We present a polynomial time algorithm in $N$ for computing \( d(t; \boldsymbol{a}) \) when \( \boldsymbol{a} \) is bounded and \( t \) is a parameter.
The proposed algorithm is rooted in the use of cyclotomic polynomials and builds upon recent results by Xin-Zhang-Zhang on the efficient computation of generalized Todd polynomials. The algorithm has been implemented in \texttt{Maple} under the name \texttt{Cyc-Denum} and demonstrates superior performance when \( a_i \leq 500 \) compared to Sills-Zeilberger's \texttt{Maple} package \texttt{PARTITIONS}.
\end{abstract}

\maketitle

\noindent
\begin{small}
 \emph{Mathematic subject classification}: Primary 05--08; Secondary 05--04, 05A17.
\end{small}

\noindent
\begin{small}
\end{small}

\section{Introduction}
For a positive integer sequence $\a = (a_1, \dots, a_{N})$ with $\gcd(\a) = 1$ and a nonnegative integer $t$,
the Sylvester's \emph{denumerant} \cite{1857-Sylvester-Wave}, denoted by $d(t; \a)$, is a function that enumerates the number of nonnegative integer solutions to  $\sum_{i=1}^{N} a_i x_i = t$. This concept has been the subject of extensive research, and a notable result states that $d(t; \a)$ is a quasi-polynomial in $t$ of degree $N-1$.

In 2015, Baldoni et al. \cite{2015-Baldoni-Denumerant} gave a polynomial time algorithm for computing the top $k$ coefficients when $k$ is fixed. Meanwhile, they developed a \texttt{Maple} package \texttt{M-Knapsack}, as well as a \texttt{C++} package \texttt{LattE knapsack}.
In 2023 \cite{2023-denumerant}, we introduced our own \texttt{Maple} package \texttt{CT-Knapsack}, which has a significant speed advantage over \texttt{M-Knapsack}. The experimental results can be found in \cite[Section 5]{2023-denumerant}.
However, both algorithms of Baldoni et al.'s and our own rely on Barvinok's unimodular cone decomposition.
The computations of the bottom coefficients of $d(t; \a)$ become challenging when $N$ is large, as it necessitates the decomposition of high dimensional cones.
This issue was addressed in our previous work \cite{2023-denumerant}.
Then we take note of Sills and Zeilberger's work in 2012 \cite{2012-Zeilberger}. They provided a \texttt{Maple} package \texttt{PARTITIONS}, which can compute $p_k(t)$ for $k$ up to $70$, where $p_k(t) = d(t; 1, 2, \dots, k)$. But it already takes $83.4$ seconds for \texttt{CT-Knapsack} to compute $p_{12}(t)$. The computation of
$p_{70}(t)$ is out of reach even for \texttt{LattE Knapsack}.

The main idea of Sills and Zeilberger's \texttt{PARTITIONS} is deriving the formula for $d(t; \a)$ with the aid of the partial fraction decomposition of the generating function. This idea dates back at least to Cayley \cite{1856-Cayley}. They ask \texttt{Maple} to convert the generating function
\begin{equation}\label{e-generatingfunc}
\sum_{t\ge 0} d(t; \a) q^t = \frac{1}{\prod_{i=1}^{N} (1 - q^{a_i})}
\end{equation}
into partial fractions. Then for each piece, \texttt{Maple} finds the first few terms of the Maclaurin expansion and then fits the data with an appropriate quasi-polynomial using \emph{undetermined coefficients}. The output is the list of these quasi-polynomials whose sum is the desired expression for $d(t; \a)$. See \cite[Section 3]{2012-Zeilberger} for details.

Our motivation for this work is to give an efficient algorithm for the $d(t; \a)$ for when $\a$ has small elements ($\le 500$) but $N$ is large.
We introduce an algorithm, called \texttt{Cyc-Denum} (See Algorithm \ref{alg-Wave}), which is capable of computing $d(t; 1, \dots, k)$ for any $1 \le k \le 117$ in $1$ minute.
The algorithm is so named since we heavily use the \emph{$f$-th cyclotomic polynomial}, which is denoted by
\[
\Phi_f(x) := \prod_{\zeta \in \P_f} (x - \zeta),
\]
where $\P_f$ denotes the set of all $f$-th primitive roots of unity, i.e.,
\[
\P_f = \{\zeta : \zeta=\e^{\frac{2j\pi\sqrt{-1}}{f}} \text{ with } 1 \le j \le f \text{ and } \gcd(j,f)=1 \}.
\]

Another algorithm, named \texttt{Float-Denum} (See Algorithm \ref{alg-float}), introduced in this paper is based on floating computations, which can compute $d(t; 1, \dots, k)$ for any $1 \le k \le 146$ in $1$ minute. In contrast, \texttt{PARTITIONS} limited to computations for $k$ to $58$ in $1$ minute. This comparison underscores the enhanced efficiency and broader scope of our two algorithms in the context of denumerant. The detailed experiments are reported in Subsection \ref{ss-expri}.

It is well-known \cite{2022-Uday} that
\begin{equation}\label{e-Denum}
d(t; \a) = \sum_{f} W_f (t; \a),
\end{equation}
where $f$ ranges over all divisors of $a_1, \dots, a_{N}$, and $W_f (t; \a)$ is called a \emph{Sylvester wave} and descried as a \emph{quasi-polynomial} in $t$ of period $f$ for each $f$.
A \emph{quasi-polynomial of period $f$} is a function which can be written as $F(t) = [P_1(t), \dots, P_f(t)]$, where $P_i(t)$'s are polynomials in $t$ and $F(t)=P_i(t)$ if $t \equiv i \pmod f$.
Note that we treat $t \equiv 0 \pmod f$ as $t \equiv f \pmod f$.
For more research on denumerant and Sylvester wave, see e.g. \cite{2003-Andrews, 2002-Agnarsson, 2018-Aguilo, 2002-Fel, 2018-O’Sullivan, 2019-O’Sullivan, 1995-Lisonek}.

\begin{exam}\label{ex-136}
Let $\a = (1,3,6)$, then
\[
d(t; \a) = W_1(t; \a) + W_2(t; \a) + W_3(t; \a) + W_6(t; \a),
\]
where

\begin{align*}
W_1(t; \a) &= \left[ \frac{t^2}{36} + \frac{5 \ t}{18} + \frac{127}{216} \right],  &
W_2(t; \a) &= \left[ -\frac{1}{24}, \frac{1}{24} \right],  \\
W_3(t; \a) &= \left[ -\frac{1}{54}, -\frac{t}{18} - \frac{29}{108}, \frac{t}{18} + \frac{31}{108}\right],  &
W_6(t; \a) &= \left[ \frac{1}{6}, \frac{1}{12}, -\frac{1}{12}, -\frac{1}{6}, -\frac{1}{12}, \frac{1}{12} \right].
\end{align*}
Taking $t=14$ gives
\[
W_1(14; \a) =\frac{2143}{216}, \quad W_2(14; \a) = \frac{1}{24}, \quad  W_3(14; \a)=-\frac{113}{108}, \quad \text{and } \; W_6(14; \a) = \frac{1}{12}.
\]
Hence $d(14; \a) = 9$. It is easy to verify that there are exactly $9$ nonnegative integer solutions to $x_1 + 3 x_2 + 6 x_3 = 14$.
\end{exam}

Generally, we have \cite{1857-Sylvester-Wave}
\begin{equation}\label{e-Wave-1}
W_f (t; \a)
= -\res_{s=0}\sum_{\zeta \in \P_f} \frac{\zeta^{-t} \e^{-ts}}{\prod_{i=1}^{N} (1 - \zeta^{a_i} \e^{a_i s})},
\end{equation}
where $\res_{s=s_0}F(s)$ denotes the residue of $F(s)$ when expanded as a Laurent series at $s=s_0$. More precisely, $\res_{s=s_0} \sum_{i \ge i_0} c_i (s-s_0)^i = c_{-1}$.

The structure of this paper is as follows. In Section \ref{s-computation},
we commence by introducing the log-exponential trick, a key technique for the computation of generalized Todd polynomials \cite{2023-Xin-GTodd}. This trick serves as a fundamental tool in the computation of $W_f(t; \a)$. Subsequently, we delineate the computation scheme for $W_f(t; \a)$ when $f$ is given, and we present Algorithm \texttt{Cyc-Denum} in detail. The complexity analysis of this scheme is provided in Theorem \ref{theo-ca}.
In Section \ref{s-float}, we give another scheme with better performance for computing $W_f(t; \a)$ by converting the $f$-th roots of unity $\zeta$'s into floating numbers. We also provide the corresponding algorithm for this scheme, along with its complexity analysis. Section \ref{s-cr} is the concluding remark.

\section{Computation of $W_f(t; \a)$}\label{s-computation}
In this section, we provide a succinct overview of the log-exponential trick, as introduced by Xin et al. \cite{2023-Xin-GTodd}, for the efficient computation of generalized Todd polynomials. We utilize this trick to compute $W_f(t; \a)$ for a fixed $f$. Additionally, we present Algorithm \texttt{Cyc-Denum}, which has been implemented in \texttt{Maple}.

\subsection{Two Tools}
The first tool we use is a straightforward lemma that serves as a fundamental result for our calculations.
\begin{lem}\label{lem-sumpunits}
Let $f$ be a positive integer, and let $F(x)$ be a rational function in $x$. Suppose $F(\zeta)$ exists for all $\zeta$ satisfying $\zeta^f=1$. Then
\[
\sum_{\zeta \in \P_f} F(\zeta) = \frac{1}{f} \sum_{\zeta: \zeta^f=1} \zeta F(\zeta) \hPhi_f(\zeta) \Phi'_f(\zeta), \qquad
\text{where }
\hPhi_f(x) = \frac{x^f-1}{\Phi_f(x)}.
\]
\end{lem}
\begin{proof}
Note that both $\hPhi_f(x)$ and $\Phi'_f(x)$ are polynomials in $x$, and $\hPhi_f(\zeta) = 0$ for any $\zeta \in \{\zeta : \zeta^f = 1\} \setminus \P_f$. A direct computation using L'Hôpital's rule yields
\begin{align*}
\frac{1}{f}\sum_{\zeta : \zeta^f = 1} \zeta F(\zeta) \hPhi_f(\zeta) \Phi'_f(\zeta)
&= \sum_{\zeta \in \P_f} \zeta F(\zeta) \hPhi_f(\zeta) \Phi'_f(\zeta) \zeta^{f-1} \lim_{x=\zeta}\frac{x - \zeta}{x^f - 1} \\
&= \sum_{\zeta \in \P_f} F(\zeta) \Phi'_f(\zeta) \lim_{x=\zeta} \hPhi_f(x) \frac{x - \zeta}{x^f - 1} \\
&=\sum_{\zeta \in \P_f} F(\zeta) \Phi'_f(\zeta) \lim_{x=\zeta} \frac{x - \zeta}{\Phi_f(x)} \\
&= \sum_{\zeta \in \P_f} F(\zeta) \Phi'_f(\zeta) \frac{1}{\Phi'_f(\zeta)} \\
&= \sum_{\zeta \in \P_f} F(\zeta),
\end{align*}
as desired.
\end{proof}

The second tool involves the computation of generalized Todd polynomials using the log-exponential trick, a concept introduced by Xin et al. in \cite{2023-Xin-GTodd}. Subsequently, we will consistently use the notation \( R = \QQ(\zeta) \), where \(\zeta\) is a primitive \(f\)-th root of unity, with \( R \) collapsing to \( \QQ \) when \(\zeta = 1 \). This is adequate for our objectives.

Consider \( a \in \QQ \) and let \( B_0, \bar B_0, B_1, \bar B_1, \dots, B_r, \bar B_r \) be finite multi-sets of nonzero integers. The generalized Todd polynomials, denoted as \( gtd_d \), are defined by their generating function:
\[
F(s) = \sum_{d\ge 0} gtd_d s^d = \e^{as} \frac{\prod_{b \in B_0} g(bs)}{\prod_{b \in \bar B_0} g(bs)} \prod_{i=1}^{r}\frac{\prod_{b \in B_i} g(bs, y_i)}{\prod_{b \in \bar B_i} g(bs, y_i)},
\]
where
\[
g(s) = \frac{s}{\e^s - 1} = 1 + o(1), \qquad g(s, y) = \frac{1}{1 - y(\e^s-1)} = 1 + o(1).
\]
For our purposes, we focus on the scenario where \( \bar B_i \) is empty and \( y_i \in R \) for all \( i \). Thus, the generating function simplifies to:
\begin{equation}\label{e-gtd}
F(s) = \sum_{d\ge 0} gtd_d s^d = \e^{as} \prod_{b \in B_0} g(bs) \prod_{i=1}^{r}\prod_{b \in B_i} g(bs, y_i).
\end{equation}

The log-exponential trick involves first calculating \( H(s) = \ln(F(s)) \) using expansion formulas:
\begin{equation}\label{e-hexpansion}
h(s) = \ln g(s) = -\sum_{k \ge 1} \frac{\B_k}{k \cdot k!} s^k, \qquad
h(s, y) = \ln g(s, y) = \sum_{k \ge 1} C_k(y) s^k,
\end{equation}
and subsequently computing \( \e^{H(s)} \). Here, the \(\B_k\)'s are the Bernoulli numbers, and \( C_k(y)\) are polynomials in $y$ that we will not use in this paper. Xin et al. \cite{2023-Xin-GTodd} have provided a thorough examination of the computations involving logarithms and exponentials, along with complexity analysis. Their analysis was conducted under the assumption that the \( y_i \)'s are variables. We revisit this complexity analysis, focusing on the scenario where \( y_i \in R \) for all \( i \).

The complexities are quantified by the number of operations in the ring \( R \).
We also require several results that have been established \cite{2023-Xin-GTodd}.

A commutative ring \( R \) is said to support fast Fourier transform (FFT) if it contains a primitive \( 2^k \)th root of unity for any nonnegative integer \( k \). We will consistently assume that FFT is feasible.
\begin{lem}[\cite{2023-Xin-GTodd}]\label{lem-log}
Let \( d= 2^\ell \) for a positive integer \( \ell \). Given \( G(s) \xmod{s^d} \in R[s] \) with \( G(0)=1 \), then \( \ln G(s) \xmod{s^{d}} \) can be computed using \( O(d \log d) \) operations in \( R \).
\end{lem}

\begin{lem}[\cite{2023-Xin-GTodd}]\label{lem-exp}
Let \( d= 2^\ell \) for a positive integer \( \ell \). Given \( H(s)\xmod{s^{d}} \in R[s] \) with \( H(0)=0 \), \( \e^{H(s)} \xmod{s^{d}} \) can be computed in \( O(d\log(d)) \) operations in \( R \).
\end{lem}

\begin{lem}[\cite{2023-Xin-GTodd}]\label{lem-add-similar}
Consider a multi-set $ B $ of $ k $ elements within the ring $ R $ as previously defined. If $ H(s) \xmod{s^d}$ is provided in $ R[s]$ with $ H(0) = 0 $, then the computation of $ \sum_{b \in B} H(b s) $ can be achieved using $ O(k \log^2(d) + d\log(d)) $ \mmm{R}.
\end{lem}

\begin{lem}\label{lem-hsandhsy}
Let $h(s)$ and $h(s,y)$ with $y \in R$ be defined as in \eqref{e-hexpansion}. Then $h(s) \xmod{s^d}$ and $h(s, y) \xmod{s^d}$ can be computed using $O(d\log(d))$ operations in $R$.
\end{lem}
\begin{proof}
The complexity of computing $h(s) \xmod{s^d}$ is stated in \cite[Lemma 21]{2023-Xin-GTodd}, and the computation of $h(s, y) \xmod{s^d}$ can be derived from
\[
h(s, y) = \ln g(s, y) = - \ln(1 - \sum_{i \ge 1} \frac{y}{i!} s^i)
\]
and Lemma \ref{lem-log}.
\end{proof}

\begin{theo}\label{theo-GTodd}
Suppose $d$ is a positive integer.
For given $a, B_0, B_1, \dots, B_r$, let $F(s)$ be as defined in \eqref{e-gtd}.
Then we can compute the sequence $(gtd_0, gtd_1, \dots, gtd_{d-1})$ of generalized Todd polynomials using $O\Big((r+1) d \log(d) + \log^2(d) \sum_{i=0}^{r} |B_i|\Big)$ \mmm{R}.
\end{theo}
\begin{proof}
Let
\[
H(s) = \ln (F(s)) = as + \sum_{b \in B_0} h(bs) + \sum_{i=1}^{r} \sum_{b \in B_i} h(bs, y_i).
\]
We proceed by constructing $H(s) \xmod{s^d}$ and subsequently computing $F(s)\equiv e^{H(s)} \xmod{ s^{d}}$, which results in the desired sequence \( (gtd_0, gtd_1, \dots, gtd_{d-1}) \).

In Step 0, we utilize Lemma \ref{lem-hsandhsy} to compute $h(s)$ using $O(d \log(d))$ \mmm{R}, and compute $h(s, y_i)$ for $1 \leq i\leq r$ using $O(r d \log(d))$ \mmm{R}.

In Step 1, we apply Lemma \ref{lem-add-similar} to compute $\sum_{b \in B_0} h(bs) \xmod{s^{d}}$ using $O\Big(d \log(d) + |B_0| \log^2(d) \Big)$ \mmm{R} and $\sum_{b \in B_0} h(bs, y_i) \xmod{s^{d}}$ for $1 \leq i\leq r$ using $O\Big(r d \log(d) + \log^2(d) \sum_{i=1}^{r} |B_i|\Big)$ \mmm{R}.

In Step 2, we invoke Lemma \ref{lem-exp} to compute $\e^{H(s)} \xmod{s^{d}}$. The operational cost for this step amounts to $O(d \log(d))$ \mmm{R}.

The total complexity is the sum of the complexities of the individual steps, leading to the claimed bound.
\end{proof}

\subsection{Computation of $W_f(t; \a)$}
Throughout this and the next subsections, we often operate within the field \( R = \QQ[x]/\langle \Phi_f(x) \rangle \).
An element \( P(x) + \langle \Phi_f(x) \rangle \) in \( R \) is represented by its \emph{standard form} \( \Rsta{ P(x)} \), which is the unique polynomial representative of degree at most \( \varphi(f)-1 \), where \( \varphi(f) = \# \{j : 1\le j \le f \text{ and } \gcd(j, f) = 1\} \) is Euler's totient function.
Specifically, \( \Rsta{ P(x)} \) is determined through two steps: i) Calculate the remainder \( Q \) of \( P \) when divided by \( x^f-1 \); ii) Calculate the remainder of \( Q \) when divided by \( \Phi_f(x) \).

The following lemma asserts that \( \frac{1}{1-\zeta} \) can be expressed as a polynomial in \( \zeta \) for any \( f \)-th root of unity \( \zeta \) (\(\neq 1\)) (hence a polynomial of standard form over \( R \) if \( \zeta \in \P_f \)).
\begin{lem}\label{lem-inver}
Suppose \( \zeta \neq 1 \) is a \( f \)-th root of unity, and \( \theta_f(x) = \sum_{i=0}^{f-1} x^i \). Then we have
\[
\frac{1}{1-\zeta} = - \frac{1}{f} \zeta \theta'_f(\zeta).
\]
\end{lem}
\begin{proof}
Observe that \( x^f - 1 = (x-1)\theta_f(x) \). Taking the derivative of both sides with respect to \( x \) yields \( f x^{f-1} = \theta_f(x) + (x-1) \theta'_f(x) \). The lemma follows from \( \theta_f(\zeta) = 0 \) for any \( \zeta \neq 1 \).
\end{proof}

For a fixed $f$ and each $\zeta \in \P_f$, $\zeta^{a_i}=1$ if and only if $f \mid a_i$. Denote by $n(f \mid \a) := \# \{i : f \mid a_i\}$. Then \eqref{e-Wave-1} can be written as
\begin{equation}\label{e-Wave-2}
\begin{aligned}
W_f (t; \a) &= -\res_{s=0}\sum_{\zeta \in \P_f} \frac{\zeta^{-t} \e^{-t s}}{\prod_{i=1}^{N} (1-\zeta^{a_i} \e^{a_is})} \\
&= -\res_{s=0} \sum_{\zeta \in \P_f} \frac{\zeta^{-t} }{\prod_{i=1}^{N} (1-\zeta^{a_i} \e^{a_is})} \sum_{m \ge 0} \frac{(-t s)^m}{m!} \\
&=  \sum_{m \ge 0} \frac{(-1)^{m+1}}{m!} t^m \res_{s=0} s^m \sum_{\zeta \in \P_f} \frac{\zeta^{-t} }{\prod_{i=1}^{N} (1-\zeta^{a_i} \e^{a_is})} \\
&= \sum_{m = 0}^{n(f \mid \a)-1} \frac{(-1)^{n(f \mid \a) + m+1}}{m! \prod_{i: f \mid a_i} a_i} t^m [s^{n(f \mid \a)-1-m}] \sum_{\zeta \in \P_f} \W_f(t; \a, \zeta),
\end{aligned}
\end{equation}
where
\[
\W_f(t; \a, \zeta) = \zeta^{-t} \prod_{i: f \nmid a_i} v_i(\zeta) \prod_{i : f \mid a_i} g(a_i s) \prod_{i: f \nmid a_i} g(a_i s, v_i(\zeta)-1),
\]
$v_i(x) = \Rsta{ - \frac{1}{f} x^{a_i} \theta'_f(x^{a_i})}$ is obtained by Lemma \ref{lem-inver} and the fact that $\Phi_f(\zeta) = 0$ for any $\zeta \in \P_f$,
and $[s^{n(f \mid \a)-1-m}] \W_f(t; \a, \zeta)$ must be equal to $0$ when $m \ge n(f \mid \a)$ since $\W_f(t; \a, \zeta)$ is a power series in $s$.

Note that we need to compute such $v_i(x)$ at most $f-1$ times. For any positive integer $a$ satisfying $f \nmid a$, we have $\zeta^{a} = \zeta^{a \pmod f}$.
Therefore, we only need to compute $\Rsta{- \frac{1}{f} x^{a} \theta'_f(x^{a})} $ for some $1 \le a \le f-1$.

The problem simplifies to the computation of \( [s^{n(f \mid \a)-1-m}] \sum_{\zeta \in \P_f} \W_f(t; \a, \zeta) \) for all \( 0 \le m \le n(f \mid \a) - 1 \).
The most straightforward scenario occurs when \( f = 1 \). In this case, \( n(1 \mid \a) = N \), \( \P_1 = \{1\} \), and
\[
\W_1(t; \a, 1) = \prod_{i=1}^{N} g(a_i s).
\]
Applying Theorem \ref{theo-GTodd} to compute
\begin{equation}\label{e-Wavefeq1-GTodd}
\W_1(t; \a, 1) \pmod{\langle s^{N} \rangle} = \sum_{k=0}^{N-1} A_k s^k, \quad \text{where} \quad A_k \in \QQ.
\end{equation}
Substituting the above into \eqref{e-Wave-2} yields
\begin{equation}\label{e-Wavefeq1}
W_1(t; \a) = \Big[ \sum_{m = 0}^{N-1} \frac{(-1)^{N + m+1}}{m! \prod_{i=1}^N a_i} A_{N-1-m} t^m \Big].
\end{equation}

In the case when \( f > 1 \), we apply Theorem \ref{theo-GTodd} to determine
\begin{equation}\label{e-applyGTodd}
\prod_{i : f \mid a_i} g(a_i s) \prod_{i: f \nmid a_i} g(a_i s, v_i(\zeta)-1) \pmod{\langle s^{n(f \mid \a)} \rangle} = \sum_{k = 0}^{n(f \mid \a)-1} \hM_k(\zeta) s^k,
\end{equation}
where \( \hM_k(x) \) is a polynomial in \( x \) for each \( k \). Consequently, we obtain
\begin{equation}\label{e-s-1-m}
[s^{n(f \mid \a)-1-m}] \W_f(t; \a, \zeta) = \zeta^{-t} M_m(\zeta),
\end{equation}
where
\begin{equation}\label{e-Mm}
M_m(x) = \BRsta{\hM_{n(f \mid \a)-1-m}(x) \prod_{i: f \nmid a_i} v_i(x)}.
\end{equation}
\begin{theo}\label{theo-main}
Under the aforementioned notations, for each \( 0 \le m \le n(f \mid \a)-1 \), assume
\begin{equation}\label{e-finalstep}
M_m(x) \hPhi_f(x) \Phi'_f(x) \equiv \sum_{i=0}^{f-1} c_{m,i} x^i \pmod{\langle x^f-1 \rangle}.
\end{equation}
Then
\begin{equation}\label{e-Wave-m}
[s^{n(f \mid \a)-1-m}] \sum_{\zeta \in \P_f} \W_f(t; \a, \zeta) = c_{m, \hat t-1}, \quad \text{ where } \quad \hat{t} \equiv t \pmod f.
\end{equation}
Furthermore, by substituting \eqref{e-Wave-m} into \eqref{e-Wave-2}, we find that \( W_f(t; \a) = [P_1(t), \dots, P_f(t)] \), where
\begin{equation}\label{e-Wavef-P}
P_i(t) = \sum_{m = 0}^{n(f \mid \a)-1} \frac{(-1)^{n(f \mid \a) + m+1}}{m! \prod_{i: f \mid a_i} a_i} c_{m, i-1} t^m.
\end{equation}
\end{theo}
\begin{proof}
A straightforward computation yields
\begin{align*}
[s^{n(f \mid \a)-1-m}] \sum_{\zeta \in \P_f} \W_f(t; \a, \zeta) &= \sum_{\zeta \in \P_f} \zeta^{-t} M_m(\zeta) \quad \text{(by \eqref{e-s-1-m})} \\
&= \frac{1}{f} \sum_{\zeta: \zeta^f=1} \zeta \zeta^{-t} \sum_{i=0}^{f-1} c_{m,i} \zeta^i \quad \text{(by Lemma \ref{lem-sumpunits})}\\
&= \frac{1}{f} \sum_{i=0}^{f-1} c_{m,i} \sum_{\zeta: \zeta^f=1} \zeta^{i+1-\hat{t}}
= \frac{1}{f} c_{m,\hat{t}-1} \sum_{\zeta: \zeta^f=1} \zeta^0
= c_{m,\hat{t}-1},
\end{align*}
where the penultimate equation holds because \( \sum_{\zeta: \zeta^f=1} \zeta^j = 0 \) for any \( j \pmod{f} \neq 0\).
\end{proof}

\begin{exam}
Returning to Example \ref{ex-136}, recall $\a = (1,3,6)$. 

We provide the details of the computation of $W_3(t; \a)$ to explain our computation scheme. It is obvious that $n(3 \mid \a) = 2$. By lemma \ref{lem-inver},
\[
\frac{1}{1-\zeta} = \BRsta{-\frac{1}{3} x (2 x + 1)} \Big|_{x=\zeta} = \frac{\zeta+2}{3}
\]
holds for any $\zeta \in \P_3$. According to \eqref{e-Wave-2},
\[
W_3 (t; \a) = \sum_{m = 0}^{1} \frac{(-1)^{m+1}}{18\ m!} t^m [s^{1-m}] \sum_{\zeta \in \P_3} \W_3(t; \a, \zeta),
\]
where
\[
\W_3(t; \a, \zeta) = \zeta^{-t}\ \frac{\zeta+2}{3}\ g(3 s)\ g(6 s)\  g\Big(s, \frac{\zeta - 1}{3}\Big).
\]
We apply Theorem \ref{theo-GTodd} to determine
\[
g(3 s)\ g(6 s)\  g\Big(s, \frac{\zeta - 1}{3}\Big) \pmod{\langle s^2 \rangle} = 1 + \frac{2 \zeta - 29}{6} s.
\]
By \eqref{e-Mm},
\begin{align*}
M_0(x) &= \frac{x + 2}{3}\ \frac{2\ x - 29}{6} \pmod{\langle \Phi_3(x) \rangle} = -\frac{3\ x}{2} - \frac{10}{3}, \\
M_1(x) &= \frac{x + 2}{3}\cdot 1 \pmod{\langle \Phi_3(x) \rangle} = \frac{x + 2}{3}.
\end{align*}
Furthermore, we have $\hPhi_3(x) = \frac{x^3-1}{\Phi_3(x)} = x-1$ and $\Phi'_3(x) = 2 x+1$. Then we can compute
\begin{align*}
M_0(x) \hPhi_3(x) \Phi'_3(x) &\equiv - \frac{31\ x^2}{6} + \frac{29\ x}{6} + \frac{1}{3} \pmod{\langle x^3-1 \rangle}, \\
M_1(x) \hPhi_3(x) \Phi'_3(x) &\equiv x^2 - x \pmod{\langle x^3-1 \rangle}.
\end{align*}
Finally, by Theorem \ref{theo-main}, $W_3(t; \a) = [P_1(t), P_2(t), P_f(t)]$, where
\begin{align*}
  P_1(t) &= \frac{-1}{0! \times 18} \cdot \frac{1}{3} + \frac{1}{1! \times 18} t \cdot 0 = - \frac{1}{54}, \\
  P_2(t) &= \frac{-1}{0! \times 18} \cdot \frac{29}{6}  + \frac{1}{1! \times 18} t \cdot (- 1) = -\frac{t}{18} - \frac{29}{108}, \\
  P_3(t) &= \frac{-1}{0! \times 18} \cdot \Big(- \frac{31}{6} \Big) + \frac{1}{1! \times 18} t \cdot 1 = \frac{t}{18} + \frac{31}{108}.
\end{align*}
The result is in perfect agreement with that presented in Example \ref{ex-136}.
\end{exam}

\subsection{Algorithm and complexity analysis}
We present Algorithm \texttt{Cyc-Denum} to encapsulate the computation of \( W_f(t; \a) \). This algorithm has been implemented in the namesake \texttt{Maple} package and its complexity is asserted in Theorem \ref{theo-ca}.

\begin{algorithm}[H]
\DontPrintSemicolon
\KwInput{A positive integer sequence $\a$ with $\gcd(\a)=1$. \\
        \hspace{3.2em} A (symbolic) nonnegative integer $t$.}
\KwOutput{The list of $W_f(t; \a)$ as described in \eqref{e-Denum}.}

Let $S$ be the set of all divisors of the entries of $\a$.

\For{$f \in S$}{

    \If{$f=1$}{

        Compute \eqref{e-Wavefeq1-GTodd} and obtain $W_1(t; \a)$ by \eqref{e-Wavefeq1}.

    }

    \Else{

        Compute $v_i(x) = \Rsta{- \frac{1}{f} x^{a_i} \theta'_f(x^{a_i})}$ for all $a_i$ satisfying $f \nmid a_i$, and derive \eqref{e-applyGTodd} using Theorem \ref{theo-GTodd}.

        For each $0 \le m \le n(f \mid \a)-1$, compute $M_m(x)$ as in \eqref{e-Mm} and $\sum_{i=0}^{f-1} c_{m,i} x^i$ as in \eqref{e-finalstep}.

        Obtain $W_f(t; \a) = [P_1(t), \dots, P_f(t)]$, where $P_i(t)$'s are as in \eqref{e-Wavef-P}.

    }
}

\Return The list of $W_f(t; \a)$ for all $f \in S$.

\caption{Algorithm \texttt{Cyc-Denum}}\label{alg-Wave}
\end{algorithm}

\begin{theo}\label{theo-ca}
Let $\a = (a_1,\dots,a_{N})$ be a sequence of positive integers with $\gcd(\a)=1$ and $f$ be a factor of certain $a_i$.

\begin{enumerate}
  \item Algorithm \texttt{Cyc-Denum} correctly computes $W_1(t; \a)$ using $O(N \log^2(N))$ \mmm{\QQ}.

  \item For $f>1$, Algorithm \texttt{Cyc-Denum} correctly computes $W_f(t; \a)$ using $O(d)$ \mmm{\QQ[x]/\langle x^f - 1 \rangle} and $O\big(f d\log(d) + N \log^2(d)\big)$ \mmm{R}, where $d=n(f \mid \a)$.
\end{enumerate}

\end{theo}

\begin{proof}
In Step 0, we compute $\prod_{i: f \mid a_i} a_i$ and $m!$ for $0 \le m \le n(f \mid \a)-1$ using $O(d)$ \mmm{\QQ} (or $O(N)$ \mmm{\QQ} if $f=1$).

For part (1), we proceed to compute \eqref{e-Wavefeq1-GTodd} by Theorem \ref{theo-GTodd} using $O( N\log^2(N) )$ \mmm{\QQ} and obtain $W_1(t; \a)$ by \eqref{e-Wavefeq1} using $O(N)$ \mmm{\QQ}. Then the total complexity is clearly $O(N \log^2(N))$ \mmm{\QQ}.

For part (2), we conclude the computation with the following steps.

In Step 1, we compute $v_i(x) = \Rsta{- \frac{1}{f} x^{a_i} \theta'_f(x^{a_i})}$ for all $i \in \{i : f \nmid a_i\}$ using $O(f)$ \mmm{R}.

In Step 2, we compute $\Rsta{\prod_{i: f \nmid a_i} v_i(x)}$ using $O(N -d)$ \mmm{R}.

In Step 3, we utilize Theorem \ref{theo-GTodd} to compute \eqref{e-applyGTodd}. This step uses at most
$O\big(f d\log(d) + N \log^2(d)\big)$ \mmm{R}.

In Step 4, we compute $M_m(x)$ as in \eqref{e-Mm} for all $0 \le m \le d-1$ using $O(d)$ \mmm{R}.

In Step 5, we compute $\hPhi_f(x) \Phi'_f(x)$ and \eqref{e-finalstep} for all $0 \le m \le d-1$ using $O(d)$ \mmm{\QQ[x]/\langle x^f - 1 \rangle}.

In Step 6, we derive $P_i(t)$ as in \eqref{e-Wavef-P} for all $1 \le i \le f$ using $O(f)$ \mmm{\QQ}.

Then the total complexity is the sum of the complexities of all the above steps, yielding the asserted bound.
\end{proof}

\subsection{Experiments}\label{ss-expri}

In this subsection, we report the results of our computational experiments, which are all achieved on the same personal laptop.

One experiment is computing $d(t; 1, 2, \dots, k)$ by algorithms \texttt{PARTITIONS}, \texttt{Cyc-Denum} and \texttt{Float-Denum} (as described in Section \ref{s-float}), respectively. We only retain the data whose running time is less than $1$ minute and plot the line plots (See Figure \ref{fig-PartCyfl}).
\texttt{PARTITIONS} (the green line) can compute $d(t; 1, 2, \dots, k)$ in $1$ minute for any $k \le 58$. \texttt{Cyc-Denum} (the blue line) can do so for $ k \le 117$, and \texttt{Float-Denum} (the yellow line) for $k \le 146$.
The running time ratios of these algorithms are visualized in Figure \ref{fig-ratio}.
We must say that the ratios for small $k$ should be considered invaluable because the computation time of all the above three algorithms is close to $0$ seconds.

\begin{figure}[htp!]
  \centering
  \includegraphics[width=16cm]{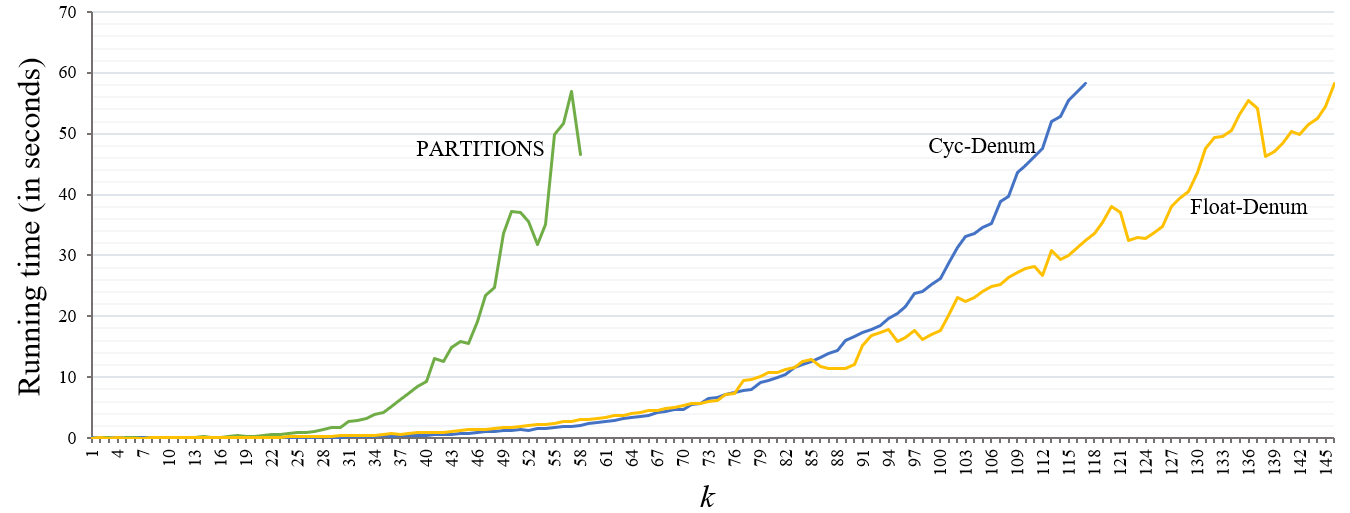}
  \caption{Running time (in seconds) for computing $d(t; 1, 2, \dots, k)$. }\label{fig-PartCyfl}
\end{figure}

\begin{figure}[htp!]
\centering
\subfigure[$\texttt{PARTITIONS}/ \texttt{Cyc-Denum}$]{
\label{fig-ratio1}
\includegraphics[width=7.5cm, height = 4cm]{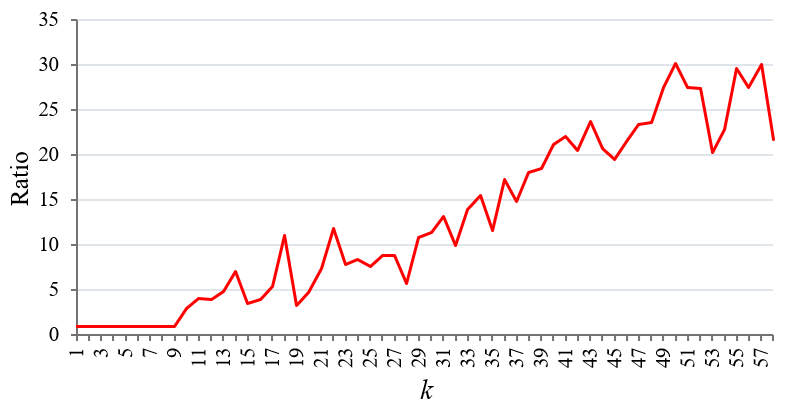}} \quad
\subfigure[$\texttt{Cyc-Denum} / \texttt{Float-Denum}$]{
\label{fig-ratio2}
\includegraphics[width=7.5cm, height = 4cm]{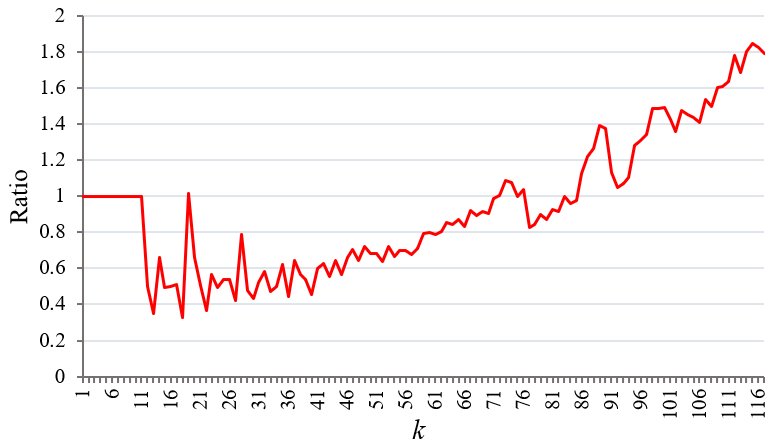}}
\caption{The ratio of running time for computing $d(t; 1, 2, \dots, k)$ between different algorithms.}
\label{fig-ratio}
\end{figure}

Another experiment is computing $d(t; \a)$ for random sequences $\a$ with $1 \le a_i \le 500 \ (1 \le i \le N)$. We increase $N$ sequentially from $N =2$ and compute $100$ random sequences for each $N$.
Because of the high memory of partial fraction decomposition, \texttt{PARTITIONS} does not perform well in this experiment. For instance, the computation of $\a = (428, 393)$ is out of reach for \texttt{PARTITIONS}, while \texttt{Cyc-Denum} completed the computation in \( 0.156 \) seconds.
Therefore, we focus on the experimental data of \texttt{Cyc-Denum}.
It successfully computes the denumerant of each $\a$ in $1$ minute for $N \le 55$ and in $2$ minutes for $N \le 63$. The number of successful computations in $1$ and $1.5$ minutes for $56 \le N \le 63$ are shown in Table \ref{table-random}.

\begin{center}
\begin{table}[htp]
\caption{The number of successful computations in $1$ and $1.5$ minutes for $56 \le N \le 63$.}\label{table-random}
  \begin{tabular}{c|ccccccccc}
    \toprule
    $N$ &  $56$ & $57$ & $58$ & $59$ & $60$ & $61$ & $62$ & $63$ &  \\
    \midrule
    \makecell{The number of successful \\ computations in $1$ minute} & $97$ & $97 $ & $ 85$ & $77 $ & $58$ & $45$ & $26$ & $19$  \\
    \midrule
    \makecell{The number of successful \\ computations in $1.5$ minutes} & $100 $ & $100 $ & $100 $ & $100 $ & $100 $ & $97 $ & $95 $ & $ 89$  \\
    \bottomrule
  \end{tabular}
\end{table}
\end{center}

\section{Floating computation}\label{s-float}
Another idea for computing $d(t; \a)$ is converting each root of unity $\zeta$ to a floating number.
Then all computations are over complex field $\CC$.
Although this method does not yield precise results, the low memory requirements enable the prediction of $d(t; \a)$ when $N$ is significantly large.

For a fixed $f$, let $\eta_f \in \CC$ be the floating value of $\e^{\frac{2 \pi \sqrt{-1}}{f}}$. Rewrite
\begin{equation}\label{e-Wave-3}
W_f (t; \a) = \sum_{m = 0}^{n(f \mid \a)-1} \frac{(-1)^{n(f \mid \a) + m+1}}{m! \prod_{i: f \mid a_i} a_i} t^m [s^{n(f \mid \a)-1-m}] \sum_{j \in J_f} \V_f(t; \a, j),
\end{equation}
where $J_f := \{j : 1\le j \le f \text{ and } \gcd(j,f)=1\}$, and
\[
\V_f(t; \a, j) = \frac{\eta_f^{- j t}}{\prod_{i: f \nmid a_i} (1 - \eta_f^{j a_i})} \prod_{i : f \mid a_i} g(a_i s) \prod_{i: f \nmid a_i} g\Big(a_i s, \frac{\eta_f^{j a_i}}{1 - \eta_f^{j a_i}}\Big).
\]

In fact, Theorem \ref{theo-GTodd} is applicable to the case when $\QQ \subseteq R$ (hence to $R= \CC$)  (See \cite{2023-Xin-GTodd}).
Then for each $j \in J_f$, we can use Theorem \ref{theo-GTodd} to figure out
\begin{equation}\label{e-applyGTodd-fl}
\prod_{i : f \mid a_i} g(a_i s) \prod_{i: f \nmid a_i} g\Big(a_i s, \frac{\eta_f^{j a_i}}{1 - \eta_f^{j a_i}}\Big) \pmod{\langle s^{n(f \mid \a)} \rangle} = \sum_{k = 0}^{n(f \mid \a)-1} A_k(f, j) s^k,
\end{equation}
where $A_k(f, j) \in \CC$ is a constant complex number with respect to $f$ and $j$ for each $k$. And
\begin{equation}\label{e-mcoe-fl}
[s^{n(f \mid \a)-1-m}] \V_f(t; \a, j) = \frac{ A_{n(f \mid \a)-1-m}(f, j)}{\prod_{i: f \nmid a_i} (1 - \eta^{j a_i})} \Big[ \eta^{-j}, \eta^{-2 j}, \dots, \eta^{-j f} \Big]
\end{equation}
is a quasi-polynomial in $t$. Substituting the above into \eqref{e-Wave-3} gives $W_f(t; \a) = [P_1(t), \dots, P_f(t)]$, where
\begin{equation}\label{e-Wavef-P-fl}
P_i(t) = \frac{1}{\prod_{i: f \mid a_i} a_i} \sum_{m = 0}^{n(f \mid \a)-1} \frac{(-1)^{n(f \mid \a) + m+1}}{m!} t^m \sum_{j \in J_f} \frac{ \eta^{-i j} A_{n(f \mid \a)-1-m} (f, j)(\zeta)}{\prod_{i: f \nmid a_i} (1 - \eta^{j a_i})}.
\end{equation}

The following algorithm clarifies the procedure for computing $d(t; \a)$ by floating computation scheme. Additionally, Theorem \ref{theo-flca} provides the complexity analysis of the computation of $W_f(t; \a)$ under this scheme.

\begin{algorithm}[H]
\DontPrintSemicolon
\KwInput{A positive integer sequence $\a$ with $\gcd(\a)=1$. \\
        \hspace{3.2em} A (symbolic) nonnegative integer $t$.}
\KwOutput{The list of $W_f(t; \a)$ as described in \eqref{e-Denum}.}

Let $S$ be the set of all divisors of the entries of $\a$.

\For{$f \in S$}{

    \For{$j \in J_f$}{

        Compute $\frac{1}{1 - \eta^{j a_i}} $ (hence $\frac{\eta^{j a_i}}{1 - \eta^{j a_i}}$) for all $a_i$ satisfying $f \nmid a_i$, and obtain \eqref{e-applyGTodd-fl} by Theorem \ref{theo-GTodd}.

        Compute $\eta^{j i}$ for $1\le i \le f$ and obtain $[s^{n(f \mid \a)-1-m}] \V_f(t; \a, j)$ by \eqref{e-mcoe-fl} for each $0 \le m \le n(f \mid \a)-1$.

    }

    Obtain $W_f(t; \a) = [P_1(t), \dots, P_f(t)]$, where $P_i(t)$'s are as in \eqref{e-Wavef-P-fl}.

}

\Return The list of $W_f(t; \a)$ for all $f \in S$.

\caption{Algorithm \texttt{Folat-Denum}}\label{alg-float}
\end{algorithm}

\begin{theo}\label{theo-flca}
Let $\a = (a_1, \dots, a_{N})$ be a positive integer sequence satisfying $\gcd(\a)=1$ and $f$ be a factor of certain $a_i$.
Algorithm \texttt{Folat-Denum} correctly computes $W_f(t; \a)$ using $O\big( (f d\log(d) +  N \log^2(d)) \varphi(f) \big)$ \mmm{\CC} by, where $d = n(f\mid \a)$.
\end{theo}
\begin{proof}
We complete the computation by the following steps.

In Step 0, we compute $\prod_{i : f \mid a_i} a_i$ and $m!$ for all $0 \le m \le n(f \mid \a)-1$ using $O(d)$ \mmm{\CC}.

In Step 1, we compute $\eta_f^i$ and $\frac{1}{1-\eta_f^i}$ (hence $\frac{\eta_f^i}{1-\eta_f^i}$) for $1 \le i \le f-1$ using $O(f)$ \mmm{\CC}. Then we can read the values of either $\eta_f^{a}$ or $\frac{1}{1-\eta_f^a}$ for any positive integer $a$ satisfying $f \nmid a$ since $\eta_f^{a}=\eta_f^{a \pmod{f}}$.

In Step 2, we compute $\frac{1}{\prod_{i: f \nmid a_i} (1 - \eta^{j a_i})}$ for all $j \in J_f$ using $O(\varphi(f) (N - d))$ \mmm{\CC}.

In Step 3, we use Theorem \ref{theo-GTodd} to obtain \eqref{e-applyGTodd-fl} using $O\big( (f d\log(d) +  N \log^2(d)) \varphi(f) \big)$ \mmm{\CC}.

In Step 4, for all $j \in J_f$, we compute $[s^{n(f \mid \a)-1-m}]\V_f(t; \a, j)$ for $1 \le t \le f$ using $O(f \varphi(f) )$ \mmm{\CC}.

In Step 5, we substitute all $[s^{n(f \mid \a)-1-m}]\V_f(t; \a, \zeta)$ into \eqref{e-Wave-3} to obtain $W_f(t; \a)$ using $O(d)$ \mmm{\CC}.

The total complexity is the sum of all the above steps.
\end{proof}

\begin{exam}
The results of Example \ref{ex-136} by Algorithm \texttt{Float-Denum} (by setting digits $=10$) are as follows:
\begin{align*}
W_1(t; \a) \approx [&0.02777777778\ t^2 + 0.2777777778\ t + 0.5879629633]; \\
W_2(t; \a) \approx [&-0.04166666667, 0.04166666667] \\
W_3(t; \a) \approx [&-3.182178376 \times 10^{-12}\ t - 0.01851851856, -0.05555555556\ t - 0.2685185186, \\ &0.05555555554\ t + 0.2870370370] \\
W_6(t; \a) \approx [&0.1666666667, 0.08333333340, -0.08333333332, -0.1666666666, \\
&-0.08333333330, 0.08333333328].
\end{align*}
Now taking $t=14$, we will obtain $d(t; \a) \approx 9.000000001$. It differs from the exact value $9$ by $10^{-9}$.
Readers can verify other values.

\end{exam}

\begin{rem}
The errors of Algorithm \texttt{Float-Denum} are related to the digits we set. For instance, the exact value of $d(1789682; 1,3,6)$ is equal to $88971554961$.
If we set digits $=10$, then we obtain $d(1789682; 1,3,6) \approx 8.897155496\times 10^{10}$. The error reaches $1$.
But if we set digits $=20$, then we will obtain $d(1789682; 1,3,6) \approx 8.8971554961000000001 \times 10^{10}$ with error $10^{-9}$, but the running time increases.
\end{rem}

\section{Concluding remark}\label{s-cr}
We have developed a fast algorithm for the Sylvester wave when $a_i\leq 500$ for all $i$. 
From Theorem \ref{theo-flca} one sees that when $a_i\leq C$ for a constant $C$, then the Sylvester wave can be computed in polynomial
in $N$. This is a nice addition to Baldoni et al.'s result (only) on top coefficients.

Algorithm \texttt{Cyc-Denum} can be adapted to give a fast algorithm for the partial fraction decomposition of a proper rational function whose denominator
is the same as that in \eqref{e-generatingfunc}. 

The idea is also adapted to give a fast algorithm for Ehrhart series in an upcoming paper \cite{2024-Ehrhart}. 

\medskip
\noindent \textbf{Acknowledgments:}
This work was supported by the National Natural Science Foundation of China (No. 12071311).

\end{document}